\documentclass[11pt]{article}
\usepackage{amsmath,amssymb}

\newtheorem{propo}{{\bf Proposition}}[section]
\newtheorem{coro}[propo]{{\bf Corollary}}
\newtheorem{lemma}[propo]{{\bf Lemma}} \newtheorem{theor}[propo]{{\bf
Theorem}} \newtheorem{ex}{{\sc Example}}[section]

\newenvironment{proof}{{\bf Proof.}}{$\Box$}

\begin{document}

\vspace*{1.0in}

\begin{center} SOLVABLE COMPLEMENTED LIE ALGEBRAS  
\end{center}
\bigskip

\begin{center} DAVID A. TOWERS 
\end{center}
\bigskip

\begin{center} Department of Mathematics and Statistics

Lancaster University

Lancaster LA1 4YF

England

d.towers@lancaster.ac.uk 
\end{center}
\bigskip

\begin{abstract} In this paper a characterisation is given of solvable complemented Lie algebras. They decompose as a direct sum of abelian subalgebras and their ideals relate nicely to this decomposition. The class of such algebras is shown to be a formation whose residual is the ideal closure of the prefrattini subalgebras. 
\par 
\noindent {\em Mathematics Subject Classification 2000}: 17B05, 17B20, 17B30, 17B50.
\par
\noindent {\em Key Words and Phrases}: Lie algebras, complemented, solvable, Frattini ideal, prefrattini subalgebra, residual, Lie $A$-algebra. 
\end{abstract}

\section{Prefrattini subalgebras}
Throughout, $L$ will denote a finite-dimensional solvable Lie algebra over a field $F$. We define the {\em nilpotent residual}, $L^{\infty}$, of $L$ be the smallest ideal of $L$ such that $L/L^{\infty}$ is nilpotent. Clearly this is the intersection of the terms of the lower central series for $L$. The {\em derived series} for $L$ is the sequence of ideals $L^{(i)}$ of $L$ defined by $L^{(0)} = L$, $L^{(i+1)} = [L^{(i)},L^{(i)}]$ for $i \geq 0$; we will also write $L^2$ for $L^{(1)}$. If $L^{(n)} = 0$ but $L^{(n-1)} \neq 0$ we say that that $L$ has {\em derived length} $n$. We say that $L$ is {\em completely solvable} if $L^2$ is nilpotent. Algebra direct sums will be denoted by $\oplus$, whereas vector space direct sums will be denoted by $\dot{+}$. 
\par

The {\em Frattini subalgebra} of $L$, $\phi(L)$, is the intersection of the maximal subalgebras of $L$. When $L$ is solvable this is always an ideal of $L$, by \cite[Lemma 3.4]{bg}. For a subalgebra $U$ of $L$ we denote by $[U:L]$ the set of all subalgebras $S$ of $L$ with $U \subseteq S \subseteq L$. We say that $[U:L]$ is {\em complemented} if, for any $S \in [U:L]$ there is a $T \in [U:L]$ such that $S \cap T = U$ and $<S,T> = L$. We denote by $[U:L]_{max}$ the set of maximal subalgebras in $[U:L]$; that is, the set of maximal subalgebras of $L$ containing $U$. 
\par

Let 
\[ 0 = A_0 \subset A_1 \subset \ldots \subset A_n = L  
\] be a chief series for $L$. We say that $A_i/A_{i-1}$ is a {\em Frattini} chief factor if $A_i/A_{i-1} \subseteq \phi(L/A_{i-1})$; it is {\em complemented} if there is a maximal subalgebra $M$ of $L$ such that $L = A_i + M$ and $A_i \cap M = A_{i-1}$. When $L$ is solvable it is easy to see that a chief factor is Frattini if and only if it is not complemented. 
\par

We define the set $\mathcal{I}$ by $i \in \mathcal{I}$ if and only if $A_i/A_{i-1}$ is not a Frattini chief factor of $L$. For each $i \in \mathcal{I}$ put
\[ \mathcal{M}_i = \{ M \in [A_{i-1}, L]_{max} \colon A_i \not \subseteq M\}.
\]
Then $B$ is a {\em prefrattini} subalgebra of $L$ if 
\[ B = \bigcap_{i \in \mathcal{I}} M_i \hbox{ for some } M_i \in \mathcal{M}_i.
\]
It was shown in \cite{prefrat} that the definition of prefrattini subalgebras does not depend on the choice of chief series.
\par

The subalgebra $B$ {\em avoids} $A_i/A_{i-1}$ if $B \cap A_i = B \cap A_{i-1}$; likewise, $B$ {\em covers} $A_i/A_{i-1}$ if $B + A_i = B + A_{i-1}$. Let $\Pi(L)$ be the set of prefrattini subalgebras of $L$. Then the following results were established in \cite{prefrat}.

\begin{theor}\label{t:prefrat} Let $L$ be a solvable Lie algebra over a field $F$.
\begin{itemize}
\item[(i)] If $B$ is a prefrattini subalgebra of $L$ then it covers all Frattini chief factors of $L$ and avoids the rest.
\item[(ii)] If $B$ is a prefrattini subalgebra of $L$ then 
\[ \dim B = \sum_{i \notin {\mathcal I}}(\dim A_i - \dim A_{i-1});
\]
in particular, all prefrattini subalgebras of $L$ have the same dimension.
\item[(iii)] If $A$ is an ideal of $L$ and $S \in \Pi(L)$ then $(S + A)/A \in \Pi(L/A)$.
\item[(iv)] $ \phi(L) = \bigcap_{B \in \Pi(L)} B. $
\item[(v)] If $L$ is completely solvable then $\Pi(L) = \{\phi(L)\}$.
\item[(vi)] Suppose that $F$ has characteristic $p$ and  that $L^{\infty}$ has nilpotency class less than $p$. Then the elements of $\Pi(L)$ are conjugate under inner automorphisms of the form exp(ad\,$x)$ with $x \in L^{\infty}$.
\item[(vii)] $L$ is complemented if and only if $\Pi(L) = \{0\}$.
\end{itemize}
\end{theor}
\bigskip

If $L^2$ is not nilpotent then $\Pi(L)$ can contain more than one element (see \cite{prefrat}).

\section{The Prefrattini Residual}

Here we use the ideas of the previous section to re-examine {\em complemented} Lie algebras: that is, Lie algebras $L$ for which $[0:L]$ is complemented, as studied in \cite{comp}. Results for groups similar to those in the next theorem were stated by Bechtell in \cite{bech}.

\begin{theor}\label{t:comp} Let $L$ be a solvable Lie algebra over any field $F$. Then the following are equivalent:
\begin{itemize}
\item[(i)] $L$ is complemented;
\item[(ii)] The prefrattini subalgebras of $L$ are all trivial;
\item[(iii)] $L$ and all of its epimorphic images are $\phi$-free; and
\item[(iv)] $L$ splits over all of its ideals.
\end{itemize}
\end{theor}
\begin{proof} (i) $\Rightarrow$ (ii) : If $L$ is complemented then $\Pi(L) = \{0\}$, by Theorem \ref{t:prefrat} (vii).
\par

\noindent (ii) $\Rightarrow$ (iii) : Suppose that $\Pi(L) = \{0\}$, let $L/B$ be any epimorphic image of $L$, and suppose that $\phi(L/B) \neq 0$. Then there is a Frattini chief factor of $L$, $C/B$, contained in $\phi(L/B)$. But now any prefrattini subalgebra has dimension greater than or equal to $\dim (C/B)$, by Theorem \ref{t:prefrat} (ii): a contradiction. This establishes (iii). 
\par

\noindent (iii) $\Rightarrow$ (iv) : Suppose that $L$ and all of its epimorphic images are $\phi$-free. We use induction on $\dim L$. The result is clear if $\dim L = 1$. So suppose it holds for Lie algebras of dimension less than $\dim L$, and let $B$ be a non-trivial ideal of $L$. If $B$ is a minimal ideal of $L$ then the result follows from \cite[Lemma 7.2]{frat}. If $B$ is not a minimal ideal, let $A$ be a minimal ideal of $L$ contained in $B$. Then $L/A$ splits over $B/A$ by the inductive hypothesis. Thus there is a subalgebra $C$ of $L$ with $A \subseteq C$ such that $L = B + C$ and $B \cap C = A$. Moreover, there is a subalgebra $M$ of $L$ such that $L = A \dot{+} M$ by \cite[Lemma 7.2]{frat}. But now $C = A \dot{+} (M \cap C)$, whence $L = B \dot{+} (M \cap C)$, and (iv) is established.
\par

\noindent (iv) $\Rightarrow$ (i) : Suppose that $L$ splits over all of its ideals. We use induction on $\dim L$ again. The result is clear if $\dim L = 1$. So suppose it holds for Lie algebras of dimension less than $\dim L$, and let $A$ be a minimal ideal of $L$. Then $L = A \dot{+} M$ for some subalgebra $M$ of $L$. It is clear that $M \cong L/A$ splits over all of its ideals and so is complemented by the inductive hypothesis. It follows from \cite[Lemma 4]{comp} that $L$ is complemented.
\end{proof}
\bigskip

We say that $L$ is {\em elementary} if $\phi(B) = 0$ for every subalgebra $B$ of $L$. Let $A$ be a vector space of finite dimension and let $B$ be an abelian completely reducible subalgebra of ${\rm gl}(A)$. It was shown in \cite[Proposition 2.4]{elem} that the semidirect product $A \rtimes B$ is an elementary Lie algebra; we call such an algebra an {\em elementary Lie algebra of type I}. Then we have the following characterisation of completely solvable complemented Lie algebras.

\begin{theor}\label{t:compsolv} Let $L$ be a completely solvable Lie algebra. Then the following are equivalent:
\begin{itemize}
\item[(i)] $L$ is complemented;
\item[(ii)] $\phi(L) = 0$;
\item[(iii)] $L$ is elementary; and
\item[(iv)] $L \cong A \oplus E$, where $A$ is an abelian Lie algebra and $E$ is an elementary Lie algebra of type I.
\end{itemize}
\end{theor}
\begin{proof} The equivalence of (i), (ii) and (iii) is \cite[Theorem 1]{comp}. The equivalence of (iv) follows from \cite[Theorem 2.5]{elem} (the requirement of a perfect field in that result is required only to establish that an elementary algebra is completely solvable, and that is not needed here).
\end{proof}

\begin{lemma}\label{l:split} Let $L$ be a solvable Lie algebra, let $B$, $C$ be ideals of $L$ with $B \cap C = 0$, and suppose that $L/B$ and $L/C$ are complemented. Then $L$ splits over $B$ and over $C$.
\end{lemma}
\begin{proof} We show that $L$ splits over $C$. Since $L/B$ is complemented there is a subalgebra $U$ of $L$ with $B \subseteq U$ such that $L = (B + C) + U = C + U$ and $(B + C) \cap U = B$. Hence $C \cap U \subseteq C \cap (B + C) \cap U = C \cap B = 0$. 
\end{proof}
\bigskip

A class ${\mathcal H}$ of finite-dimensional solvable Lie algebras is called a {\em homomorph}
if it contains, along with an algebra $L$, all epimorphic images of $L$. A homomorph ${\mathcal H}$ is called a {\em formation} if $L/M, L/N \in {\mathcal H}$, where $M, N$ are ideals of $L$, implies that $L/M \cap N \in {\mathcal H}$. If ${\mathcal H}$ is a formation then, for every solvable Lie algebra $L$ there is a smallest ideal $R$ such that $L/R \in {\mathcal H}$; this is called the {\em ${\mathcal H}$-residual} of $L$. We denote the class of solvable complemented Lie algebras by ${\mathcal C}$. Then we have the following result.

\begin{theor}\label{t:form} ${\mathcal C}$ is a formation.
\end{theor}
\begin{proof} First note that ${\mathcal C}$ is a homomorph, by \cite[Lemma 3]{comp}. Let $B$, $C$ be distinct ideals of $L$ with $L/B$, $L/C \in {\mathcal C}$. We need to show that $L/B \cap C \in {\mathcal C}$. Without loss of generality we may suppose that $B \cap C = 0$. Let $0 < B_k < \ldots < B_1 = B$ be part of a chief series for $L$. We use induction on $k$. If $k = 1$ then $B$ is minimal ideal of $L$ and the result follows from Lemma \ref{l:split} and \cite[Lemma 4]{comp}. So suppose it holds whenever $k < n$ and that we have $k = n$. Then $B/B_n$, $(C + B_n)/B_n$ are distinct ideals of $L/B_n$ and the corresponding factor algebras are isomorphic to $L/B$ and $(L/C)/((C + B_n)/C)$ respectively. These are both complemented (by \cite[Lemma 3]{comp} in the case of the second). It follows from the inductive hypothesis that $L/B_n$ is complemented. But now $L$ is complemented by Lemma \ref{l:split} and \cite[Lemma 4]{comp}, and the result follows.
\end{proof}
\bigskip

We define the {\em Prefrattini residual} of a solvable Lie algebra $L$ to be
\[ \pi(L) = \bigcap\{B \colon B \hbox{ is an ideal of } L \hbox{ and } L/B \in {\mathcal C} \}.
\]
Clearly $\pi(L)$ is the smallest ideal of $L$ such that $L/\pi(L) \in {\mathcal C}$. It is also the ideal closure of the prefrattini subalgebras of $L$, by Theorem \ref{t:comp}.
\par

The class of solvable elementary Lie algebras is also a formation whose residual is the {\em elementary commutator}, $E(L)$ (see \cite[Theorem 5.1]{elem1}). The {\em abelian socle} of $L$, Asoc $L$, is the sum of the minimal abelian ideals of $L$. We have the following properties of $\pi(L)$.

\begin{propo}\label{p:res} Let $L$ be a solvable Lie algebra. Then
\begin{itemize}
\item[(i)] $\phi(L) \subseteq \pi(L) \subseteq E(L)$; if $L$ is completely solvable then $\phi(L) = \pi(L) = E(L)$;
\item[(ii)] if $A$ is an ideal of $L$ then $\pi(L/A) = (\pi(L) + A)/A$; in particular, $\pi(L/\phi(L)) = \pi(L)/\phi(L)$;
\item[(iii)] $\pi(L)$ is nilpotent if and only if $\pi(L) = \phi(L)$; and
\item[(iv)] if $B$ is a prefrattini subalgebra of $L$ then $\pi(L) = B + \pi(L)^{\infty}$.
\end{itemize}
\end{propo}
\begin{proof} (i) This follows from Theorem \ref{t:comp} and Theorem \ref{t:compsolv}.
\par

\noindent (ii) Let $A$ be an ideal of $L$ and put $\pi(L/A) = P/A$. Then we have that $L/P \cong (L/A)/(P/A)$ is complemented. Hence $\pi(L) + A \subseteq P$. Also $(L/A)/((\pi(L) + A)/A) \cong (L/\pi(L))/((\pi(L) + A)/\pi(L))$ is complemented, by \cite[Lemma 3]{comp}, so $P \subseteq \pi(L) + A$ and the result follows. 
\par

\noindent (iii) If $\pi(L) = \phi(L)$ then $\pi(L)$ is nilpotent, by \cite[Theorem 6.1]{frat}. Conversely let $\pi(L)$ be nilpotent. Suppose that $\phi(L) = 0$, let $N$ be the nilradical of $L$ and let $B$ be a prefrattini subalgebra of $L$. Then $N =$ Asoc $L = A_1 \oplus \ldots \oplus A_n$, where $A_i$ is a minimal ideal of $L$ for $1 \leq i \leq n$, and $L = N \dot{+} C$ for some subalgebra $C$ of $L$, by \cite[Theorems 7.3, 7.4]{frat}. Then $(A_1 \oplus \ldots \oplus A_{i+1})/(A_1 \oplus \ldots \oplus A_i)$ is a complemented chief factor of $L$ for each $1 \leq i \leq n-1$, and so is avoided by $B$. It follows that $B \cap N = 0$. But then $\pi(L) = 0$ and the converse follows from (ii).
\par

\noindent (iv) We have that $(B + \pi(L)^{\infty})/\pi(L)^{\infty}$ is a prefrattini subalgebra of $L/\pi(L)^{\infty}$ by Theorem \ref{t:prefrat} (iii). Moreover, $\pi(L/\pi(L)^{\infty}) = \pi(L)/\pi(L)^{\infty}$, by (ii) above, and this is nilpotent. It follows from (iii) that $\pi(L/\pi(L)^{\infty}) = \phi(L/\pi(L)^{\infty}) \subseteq (B + \pi(L)^{\infty})/\pi(L)^{\infty}$, by Theorem \ref{t:prefrat} (iv), whence $\pi(L) \subseteq B + \pi(L)^{\infty}$. The reverse inclusion is clear. 
\end{proof}
\bigskip

We define the {\em nilpotent series} for $L$ inductively by $N_0(L) = 0$, $N_{i+1}/N_i = N(L/N_i(L)$ for $i > 0$, where $N(L)$ denotes the nilradical of $L$. Finally we have the following characterisation of solvable complemented Lie algebras which is analogous to a result of Zacher for groups (see \cite{zach}).

\begin{theor}\label{t:phi} The solvable Lie algebra $L$ is complemented if and only if $\phi(L/N_i(L)) = 0$ for all $i \geq 1$.
\end{theor}
\begin{proof} Suppose first that $L$ is complemented. Then $L/N_i(L)$ is complemented, by \cite[Lemma 3]{comp}, and so $\phi(L/N_i(L)) = 0$, by Theorem \ref{t:comp}.
\par

Suppose conversely that $\phi(L/N_i(L)) = 0$ for all $i \geq 1$. We use induction on $\dim L$. The result is clear if $\dim L = 1$, so suppose the result holds for all solvable Lie algebras of dimension less than that of $L$. Then $L/N(L)$ is complemented, by the inductive hypothesis. Moreover, we have that $L = N(L) \dot{+} B$ for some subalgebra $B$ of $L$, and $N(L) =$ Asoc $L$, by \cite[Theorems 7.3, 7.4]{frat}. Put Asoc $L = A_1 \oplus \ldots \oplus A_n$. If $n = 1$, then $L$ splits over $A_1$ and $L/A_1$ is complemented, so $L$ is complemented, by \cite[Lemma 4]{comp}. So suppose that $n > 1$ and put $C_i = A_1 \oplus \ldots \oplus \hat{A_i} \oplus \ldots \oplus A_n$, where $\hat{A_i}$ means that $A_i$ is missing from the direct sum. Then $L/C_i$ splits over Asoc $L/C_i$ and $(L/C_i)/$(Asoc $L/C_i) \cong L/N(L)$ is complemented, so $L/C_i$ is complemented, by \cite[Lemma 4]{comp} again. It follows from Theorem \ref{t:form} that $L \cong L/\cap_{i=1}^n C_i$ is complemented. 
\end{proof} 
\bigskip

A consequence of the corresponding result for groups is that every normal subgroup of a complemented solvable group is itself complemented. The analogue of this holds for completely solvable Lie algebras, by Theorem \ref{t:compsolv}. However, the analogue does not hold for all solvable Lie algebras as the following example shows.

\begin{ex}\label{e:comp} Let $F$ be a field of characteristic $p$ and consider the Lie algebra $L = (\oplus_{i=0}^{p-1} Fe_i) \dot{+} Fx \dot{+} Fy$ with $[e_i,x] = e_{i+1}$ for $i = 0, \ldots, p-2$, $[e_{p-1},x] = e_0$, $[e_i,y] = ie_i$ for $i = 0, \ldots,p-1$, $[x,y] = x$, and all other products zero.
Then $A = \oplus_{i=0}^{p-1} Fe_i$ is the unique minimal ideal of $L$, $L$ splits over $A$ and $L/A$ is two-dimensional and so complemented. It follows from \cite[Lemma 4]{comp} that $L$ is complemented. However, $B = A \dot{+} Fx = L^2$ is an ideal of $L$, and $\phi(B) = F(x_0 + \ldots x_{p-1})$ so $B$ is not complemented. 
\end{ex}
\bigskip

\section{Decomposition results for complemented algebras}
\medskip
A Lie algebra $L$ is called an {\em $A$-algebra} if every nilpotent subalgebra of $L$ is abelian. Here we have some basic structure theorems which mirror those obtained for solvable Lie $A$-algebras in \cite{Aalg}. Where proofs are very similar to the correponding one in \cite{Aalg} we will sketch the proof for the convenience of the reader and give a reference to \cite{Aalg} for more details. First we see that $L$ splits over the terms in its derived series.
\bigskip

\begin{theor}\label{t:split} (c.f. \cite[Theorem 3.1]{Aalg})
Let $L$ be a solvable complemented Lie algebra. Then $L$ splits over each term in its derived series. Moreover, the Cartan subalgebras of $L^{(i)}/L^{(i+2)}$ are precisely the subalgebras that are complements to $L^{(i+1)}/L^{(i+2)}$ for $i \geq 0$.
\end{theor}
\begin{proof} The first assertion follows from Theorem \ref{t:comp} (iv). The second is a consequence of \cite[Theorem 4.4.1.1]{wint}.
\end{proof}
\bigskip

This gives us the following characterisation of solvable complemented Lie algebras.

\begin{coro}\label{c:decomp} 
Let $L$ be a solvable Lie algebra of derived length $n+1$. Then $L$ is complemented if and only if the following hold:
\begin{itemize}
\item[(i)] $L = A_{n} \dot{+} A_{n-1} \dot{+} \ldots \dot{+} A_0$ where $A_i$ is an abelian subalgebra of $L$ for each $0 \leq i \leq n$;
\item[(ii)] $L^{(i)} = A_{n} \dot{+} A_{n-1} \dot{+} \ldots \dot{+} A_{i}$ for each $0 \leq i \leq n$; and
\item[(iii)] $L^{(i)}/L^{(i+1)}$ is completely reducible as an $(L/L^{(i+1)})$-module for each $1 \leq i \leq n$.
\end{itemize}
\end{coro}
\begin{proof} Suppose first that $L$ is complemented. By Theorem \ref{t:split} there is a subalgebra $B_n$ of $L$ such that $L = L^{(n)} \dot{+} B_n$. Put $A_n = L^{(n)}$. Similarly $B_n \cong L/L^{(n)}$ is complemented, by \cite[Lemma 3]{comp}, so $B_n = A_{n-1} \dot{+} B_{n-1}$ where $A_{n-1} = (B_n)^{(n-1)}$. Continuing in this way we get (i). A straightforward induction proof shows (ii). Finally, $L^{(i)}/L^{(i+1)} \subseteq N(L/L^{(i+1)}) =$ Asoc$(L/L^{(i+1)})$, by Theorem \ref{t:comp} (iii) and \cite[Theorem 7.4]{frat}, which gives (iii).
\par

Suppose now that (i), (ii) and (iii) hold. Then $L$ is complemented by repeated use of \cite[Lemma 4]{comp}. 
\end{proof}
\bigskip

Next we aim to show the relationship between ideals of $L$ and the decomposition given in Corollary \ref{c:decomp}. First we need the following lemmas.

\begin{lemma}\label{l:int}
Let $L$ be a Lie algebra. Then $Z(L) \cap L^2 \subseteq \phi(L)$.
\end{lemma}
\begin{proof} Let $M$ be any maximal subalgebra of $L$. If $Z(L) \not \subseteq M$ then $L = M + Z(L)$ and $L^2 \subseteq M$.
\end{proof}

\begin{lemma}\label{l:ideal} (c.f. \cite[Lemma 3.4]{Aalg})
Let $L$ be a solvable complemented Lie algebra of derived length $\leq n+1$, and suppose that $L = B \dot{+} C$ where $B = L^{(n)}$ and $C$ is a subalgebra of $L$. If $D$ is an ideal of $L$ then $D = (B \cap D) \dot{+} (C \cap D)$.
\end{lemma}
\begin{proof} Let $L$ be a counter-example for which dim$L$ + dim$D$ is minimal. Suppose first that $D^2 \neq 0$. Then $D^2 = (B \cap D^2) \dot{+} (C \cap D^2)$ by the minimality condition. Moreover, by considering $L/D^2$ we have that $D = B \cap D \dot{+} C \cap D$, a contradiction.
We therefore have that $D^2 = 0$. Similarly, by considering $L/B \cap D$, we have that $B \cap D = 0$. 
\par
Put $E = C^{(n-1)}$. Then $(D + B)/B$ and $(E + B)/B$ are abelian ideals, and so are inside the nilradical of the complemented Lie algebra $L/B$, which is abelian. Hence
$$ [D, E] \subseteq [D + B, E + B] \subseteq B \hbox{ and } [D,E] \subseteq B \cap D = 0;$$ that is, $D \subseteq Z_L(E) = Z_B(E) + Z_C(E)$. 
\par
Now $L^{(n-1)} = B + E$, so $B = L^{(n)} = (B + E)^2 = [B,E]$. Let $L^{(n-1)} = L_0 \dot{+} L_1$ be the Fitting decomposition of $L^{(n-1)}$ relative to ad\,$E$. Then $B \subseteq L_1$ so that $Z_B(E) \subseteq L_0 \cap L_1 = 0$, whence $D\subseteq Z_L(E) = Z_C(E) \subseteq C$ and the result follows.  
\end{proof}

\begin{theor}\label{t:nz} (c.f. \cite[Theorem 3.5]{Aalg})
Let $L$ be a solvable complemented Lie algebra of derived length $n+1$ with nilradical $N$, and let $K$ be an ideal of $L$ and $A$ a minimal ideal of $L$. Then, with the same notation as Corollary \ref{c:decomp}, 
\begin{itemize}
\item[(i)] $K = (K \cap A_n) \dot{+} (K \cap A_{n-1}) \dot{+} \ldots \dot{+} (K \cap A_0)$;
\item[(ii)] $N = A_n \dot{+} (N \cap A_{n-1}) \dot{+} \ldots \dot{+} (N \cap A_0)$;
\item[(iii)] $Z(L^{(i)}) = N \cap A_i$ for each $0 \leq i \leq n$; and
\item[(iv)] $A \subseteq N \cap A_i$ for some $0 \leq i \leq n$. 
\end{itemize} 
\end{theor}
\begin{proof} (i) We have that $L = A_n \dot{+} B_n$ where $A_n = L^{(n)}$ from the proof of Corollary \ref{c:decomp}. It follows from Lemma \ref{l:ideal} that $K = (K \cap A_n) + (K \cap B_n)$. But now $K \cap B_n$ is an ideal of $B_n$, which is complemented, so $B_n = A_{n-1} \dot{+} B_{n-1}$. Applying Lemma \ref{l:ideal} again gives $K \cap B_n = (K \cap A_{n-1}) \dot{+} (K \cap B_{n-1})$. Continuing in this way gives the required result.
\par
(ii) This is clear from (i), since $A_n = L^{(n)} = N \cap A_n$.
\par
(iii) We have that $L^{(i)} = L^{(i+1)} \dot{+} A_i$ from Corollary \ref{c:decomp}, and $Z(L^{(i)}) \cap L^{(i+1)} \subseteq \phi(L^{(i)})$, by Lemma \ref{l:int}, whence $Z(L^{(i)}) \cap L^{(i+1)} \subseteq \phi(L) = 0$, using \cite[Lemma 4.1]{frat}. It follows from (i) that
$$ Z(L^{(i)}) = (Z(L^{(i)}) \cap L^{(i+1)}) + (Z(L^{(i)}) \cap A_i) = Z(L^{(i)}) \cap A_i \subseteq N \cap A_i. $$
It remains to show that $N \cap A_i \subseteq Z(L^{(i)})$; that is, $[N \cap A_i,L^{(i)}] = 0$. 
We use induction on the derived length of $L$. If $L$ has derived length one the result is clear. So suppose it holds for Lie algebras of derived length $\leq k$, and let $L$ have derived length $k+1$. Then $B = A_{k-1} + \dots + A_0 \cong L/L^{(k)}$ is a solvable complemented Lie algebra of derived length $k$, and, if $N$ is the nilradical of $L$, then $N \cap A_i$ is inside the nilradical of $B$ for each $0 \leq i \leq k-1$, so $[N \cap A_i, B^{(i)}] = 0$ for $0 \leq i \leq k-1$, by the inductive hypothesis. But $[N \cap A_i, A_k] = [N \cap A_i,L^{(k)}] \subseteq [N,N] = 0$, for $0 \leq i \leq k$, whence $[N \cap A_i,L^{(i)}] = [N \cap A_i,A_k + B^{(i)}] = 0$ for $0 \leq i \leq k$.
\par
(iv) We have $A \subseteq L^{(i)}$, $A \not \subseteq L^{(i+1)}$ for some $0 \leq i \leq n$. Now $[L^{(i)}, A] \subseteq [L^{(i)}, L^{(i)}] = L^{(i+1)}$, so $[L^{(i)}, A] \neq A$. It follows that $[L^{(i)}, A] = 0$, whence $A \subseteq Z(L^{(i)}) = N \cap A_i$, by (iii).
\end{proof}
\bigskip

A Lie algebra $L$ is called {\em monolithic} if it has a unique minimal ideal, called the {\em monolith} of $L$.

\begin{coro}\label{c:mono} Let $L$ be a solvable complemented monolithic Lie algebra of derived length $n+1$ with monolith $W$. Then $W = N = A_n = L^{(n)} = C_L(W)$, $Z(L) = 0$ and $[L,W] = W$.
\end{coro}
\begin{proof} First note that $N = A_n \dot{+} N \cap A_{n-1} \dot{+} \ldots \dot{+} N \cap A_0$ by Theorem \ref{t:nz}(i). Moreover, $N \cap A_i$ is an ideal of $L$ for each $0 \leq i \leq n-1$, by Theorem \ref{t:nz}(iii). But if $N \cap A_i \neq 0$ then $W \subseteq A_n \cap N \cap A_i = 0$ if $i \neq n$. Hence $W = N = A_n$. Also $W =$ Asoc $L = N$, by Theorem \ref{t:comp} (iii) and \cite[Theorem 7.4]{frat}, and $N = C_L(N)$ by \cite[Lemma 2.4]{Aalg}.
\par
Finally, if $Z(L) \neq 0$ then $W \subseteq Z(L) \cap L^2 = 0$, by Theorem \ref{t:nz}, a contradiction. Hence $Z(L) = 0$. It follows from this that $[L,W] \neq 0$, whence $[L,W] = W$.
\end{proof}
\bigskip

Given these shared properties between the classes of solvable Lie $A$-algebras and solvable complemented Lie algebras it is natural to ask whether either class is contained in the other. This is not the case, as the following examples show.

\begin{ex} Let $L = Fx + Fy + Fb$ with $[x,b] = x$, $[y,b] = y - x$, other products being zero. Then $\phi(L) = Fx$, so $L$ is not complemented. However, it is an $A$-algebra. For, the two-dimensional subalgebras are of the form $Fx + F(\alpha y + \beta b)$ ($\alpha, \beta \in F$), and these are nilpotent only if $\beta = 0$ and, in that case, it is abelian. 
\end{ex}

Examples of solvable complemented Lie algebras $L$ that are not $A$-algebras are a little harder to construct. In particular, if $L$ is completely solvable and complemented then it is elementary, by Theorem \ref{t:compsolv}, and so is an $A$-algebra. However, such algebras do exist in characteristic $p$ as is shown below.

\begin{ex} Let $F$ be an algebraically closed field of characteristic $p$, let $L$ be the algebra described in Example \ref{e:comp} and let $C$ be a faithful completely reducible $L$-module. Put $X = C \dot{+} L$, where $B^2 = 0$ and $L$ acts on $B$ under the given $L$-module action. Then repeated use of \cite[Lemma 4]{comp} shows that $X$ is complemented. However, $X$ is solvable of index four and so cannot be an $A$-algebra, by Drensky's Theorem (see \cite[Theorem 6.2]{Aalg}). 
\end{ex}

Notice that an easy extension of the above construction shows that, over an algebraically closed field, there are solvable complemented Lie algebras of arbitrary solvable index, whereas solvable Lie $A$-algebras over such a field have solvable index at most three.

\end{document}